\documentclass[12pt]{amsart}

\usepackage{geometry}
\usepackage{pgfplots}
\pgfplotsset{compat = 1.11}
\usetikzlibrary{calc,decorations.pathmorphing,decorations.pathreplacing,patterns}
\usepgfplotslibrary{colormaps} 
\usetikzlibrary{arrows.meta}
\tikzset{>={Latex[width=1.7mm,length=2.2mm]}}
\usetikzlibrary{decorations.text}
\usepackage{amsfonts,amssymb,amsmath,amsthm}
\usepackage{graphicx}
\usepackage{psfrag}
\usepackage{calc}
\usepackage{epic}
\usepackage{hyperref}
\usepackage{float}
\usepackage{latexsym}
\usepackage{color}
\usepackage{subfigure,dsfont}
\usepackage{todonotes}

\newtheorem{theorem}{Theorem}

\newtheorem{remark}[theorem]{Remark}

\newcommand{\be}{ \begin{equation}}
\newcommand{\ee}{\end{equation}}
\newcommand{\ben}{ \begin{equation*}}
\newcommand{\een}{\end{equation*}}

\newcommand{\Prob}{\mathbb P}
\newcommand{\Probxy}{\mathbb P_{(x,y)}}
\newcommand{\E}{\mathbb E}
\newcommand{\Exy}{\mathbb E_{(x,y)}}

\newcommand{\qxy}{q_{(x,y)}}

\newcommand{\LX}{L^X}
\newcommand{\LY}{L^Y}
\newcommand{\MX}{M^X}
\newcommand{\MY}{M^Y}

\newcommand{\Z}{\mathbb{Z}}
\newcommand{\Zpos}{\mathbb{Z}_+}

\newcommand{\R}{\mathbb{R}}

\def\1{{\mathchoice {1\mskip-4mu\mathrm l}      
{1\mskip-4mu\mathrm l}
{1\mskip-4.5mu\mathrm l} {1\mskip-5mu\mathrm l}}}

\newcommand{\nn}{\nonumber}

\title[Stability of two-dimensional Markov processes]{Stability of two-dimensional Markov processes, with an application to QBD processes with an infinite number of phases}

\author{Seva Shneer}\address{School of MACS, Heriot-Watt University, Edinburgh, EH14 4AS, United Kingdom. E-mail: V.Shneer@hw.ac.uk}
\author{Stella Kapodistria}\address{Eindhoven University of Technology, P.O. Box 513, 5600 MB Eindhoven, The Netherlands. E-mail: s.kapodistria@tue.nl}

\begin{document}
\maketitle

\begin{abstract}
In this paper, we derive a simple drift condition for the stability of a class of two-dimensional Markov processes, for which one of the coordinates (also referred to as the {\em phase} for convenience) has a well understood behaviour dependent on the other coordinate (also referred as {\em level}). The first  (phase) component's transitions  may depend on the second component and are only assumed to be eventually independent. The second (level) component  has partially bounded jumps and it is assumed to have a negative drift given that the first one is in its stationary distribution. 

The results presented in this work can be applied to processes of the QBD (quasi- birth-and-death) type on the quarter- and on the half-plane, where the phase and level are inter-dependent. Furthermore,  they provide an off-the-shelf technique to tackle stability issues for a class of two-dimensional Markov processes. These results set the stepping stones towards closing the existing gap in the literature of deriving easily verifiable conditions/criteria for two-dimensional processes with unbounded jumps and interdependence between the two components.
\end{abstract}

\section{Introduction}
In this paper, we derive a simple drift condition for the stability of a class of two-dimensional Markov processes, in which  one of the coordinates (let us refer to it as the first coordinate for convenience or as the {\em phase} coordinate) has a well-understood behaviour, and the behaviour of the other coordinate (let us refer to it as  the second coordinate or as the {\em level} coordinate) is also well-understood conditional on some specific events corresponding to the first component.

The stability condition of similar classes of two-dimensional Markov processes has received considerable attention. In \cite{foss2013stability}, the authors derive a drift condition ensuring the positive recurrence for a class of two-dimensional process in which the first component's transitions are assumed to be independent of the state of the second component and form a Harris recurrent process. The second component is such that it has a negative drift given that the first one is in its stationary distribution. In \cite{adan2020local}, it is assumed that the first component is transient, whereas the second component, conditioned on the first being large enough, has transitions of a positive recurrent Markov process. For this class, the authors derive conditions for the local stability of the second component.

We study processes similar to those of \cite{foss2013stability}. Crucially, we do not assume that the first component's transitions are independent of the state of the second component, we only assume that this happens eventually. Under further technical assumptions, we present drift conditions guaranteeing the stability of the two-dimensional process.

There is a wealth of methods and techniques for deriving stability results for Markov processes. The goal of the paper is to exploit the specific structure of the processes under consideration and using the very general framework of Lyapunov functions to develop off-the-shelf results directly applicable in particular scenarios.

A main application of our results is to QBD-type (quasi- birth-and-death) processes with a countable number of phases defined on the quarter-plane, as well as QBD-type processes on the half-plane (aka double-sided QBD processes), where our results lead to new stability criteria.

For QBD-type processes with a finite number of phases, the classical result of Neuts, that `the average drift of the level process is negative' provides a very simple stability (sufficient and necessary) condition, see \cite[Theorem 1.7.1]{neuts1981matrix}. However, when the number of phases is (countably) infinite, very little is known on how to derive the stability condition. In \cite{tweedie1982operator}, it is  shown that under an additional boundary condition, the classical stability result of Neuts can be extended to the infinite phase case. In \cite{haque2005sufficient, latouche2013level, motyer2006decay, ramaswami1996some, takahashi2001geometric}, the  authors provide sufficient conditions under which the stationary distribution for the level process of a QBD process or a (level-independent) GI/M/1-type Markov chain with countably many phases has a geometric tail. In \cite{kapodistria2017matrix}, by focusing on a class of two-dimensional Markov processes, the authors show how the classical result of Neuts on the stability condition can be adapted in the case of QBD processes with countable phases, in which transitions are only allowed to neighbouring states (viz, the tridiagonal blocks in the generator matrix themselves each have a tridiagonal structure). By exploring the underlying structure of the model, the authors proved the connection between the QBD drift condition derived by Neuts with the known drift condition for nearest neighbour random walks presented in \cite[Theorem 1.2.1]{fayolle1999random}.


Our main goal is to establish stability results for QBD-type processes on the quarter-plane and on the half-plane, thus providing a basis for extending other important results to these settings.

As in many results related to positive recurrence of Markov chains and processes, results for discrete-time processes can often by adapted to continuous-time ones. Moreover, results with simple and transparent proofs for simple state spaces may be extended to more general ones, with little additional difficulty in arguments but a considerable additional amount of notation and cumbersome derivations. We therefore choose to present a detailed proof for the simplest scenario of a discrete-time Markov chain on $\Zpos^2$, with an explanation of how this proof, in a standard way, may be generalised to a continuous-time Markov process on $\Zpos^2$. The latter result is immediately applicable to QBD processes with an infinite number of phases, which is one of the main contributions of the paper. We further provide a number of extensions of our main results which do not present any significant technical difficulties and we thus leave them without proofs. One such extension concerns QBD-type processes on the half-plane and we present a specific model where this result immediately leads to a stability criterion.

Our contribution is two-fold: First, our results lead to the stability criteria for QBD processes on the quarter-plane (in other words,  for the case of QBD processes with an infinite number of phases) and on the half-plane. Second, our results provide an off-the-shelf technique to tackle stability issues in a specific class of two-dimensional Markov processes, and we expect them to be applicable to other models of interest.


The paper is organised as follows: In Section \ref{sec:main}, the three models under consideration are described in detail: In Section \ref{subsec:discrete}, the results for the discrete time model are presented together with the four assumptions: 1) eventual independence, 2) positive recurrence of the first component, 3) boundedness of jumps, 4) Lyapunov condition. In Section \ref{subsec:continuous}, we proceed with the results for  the continuous time model and in Section \ref{subsec:qbd} we apply the results to QBD processes in the quarter-plane.  In Section \ref{sec: generalisations}, we discuss generalisations of the main results of the previous section. In particular, we assert the stability result for QBD-type processes on the half-plane and apply these results to a new model. In the following subsection we discuss further generalisations.

\subsection{Notation}
We use the following notation throughout. By $\R$ we denote the set of real number, by $\Z$ the set of integers and by $\Zpos$ the set of non-negative integers.

Our results rely on a number of technical assumptions, which are slightly different in the various settings we consider. We number these assumptions and further label them with a letter. We use `d' in the discrete-time setting, `c' in the continuous-time setting and finally `g' in the generalisation to the half-plane.

\section{Assumptions and main results} \label{sec:main}

\subsection{Discrete time} \label{subsec:discrete}

We assume that $\{(X_n,Y_n)\}_{n \ge 0}$ is a time-homogeneous, irreducible aperiodic Markov chain on $\Zpos^2$. We denote by
$\Probxy$ the measure induced by conditioning on the event $\{(X_0,Y_0) = (x,y)\}$ and by $\Exy$ the expectation with respect to this measure. For instance,
$$
\Probxy\left((X_1,Y_1) = (x',y')\right) = \Prob\left ((X_1,Y_1) = (x',y')|(X_0,Y_0) = (x,y)\right)
$$
and
$$
\Exy\left(g(X_1,Y_1)\right) = \E\left ( g(X_1,Y_1)|(X_0,Y_0) = (x,y)\right)
$$
for any function $g: \Zpos^2 \to \R$. We make the following additional assumptions.
\begin{description}
\item[Assumption 1d (eventual independence)] There exists $0 \le y^* < \infty$ such that
$$
\Probxy(X_1 = x') = \Prob^*(x,x'),
$$
which does not depend on $y$ as long as $y \ge y^*$. 
\end{description}

It is clear that the function $\Prob^*$ defined here is such that $\Prob^*(x,x') \ge 0$ for any $x, x' \in \Zpos$ and $\sum_{x'} \Prob^*(x,x') = 1$ for any $x \in \Zpos$.

\begin{description}
\item[Assumption 2d (positive recurrence of the $X$ component)] Introduce a Markov chain $\{X_n^*\}_{n \ge 0}$ via
$$
\Prob(X^*_{n+1} = x'\,|\,X^*_n = x) = \Prob^*(x,x')
$$
for any $n \ge 0$ and for any $x,x' \in \Zpos$. We assume that $\{X_n^*\}_{n \ge 0}$ is positive recurrent and denote its stationary distribution by $\pi^*$.

\item[Assumption 3d (boundedness of jumps)] There exists $0 \le H^Y < \infty$ such that $\Probxy\left ((X_1,Y_1) = (x',y')\right) = 0$ if $y' - y < -H^Y$. There also exists $0 \le H^X < \infty$ such that $\Probxy((X_1,Y_1)=(x',y')) = 0$ if $x'-x > H^X$ and $y < y^*$.
\end{description}

Assumption 3d guarantees that jumps to the right and down are bounded with probability $1$.

\begin{description}

\item[Assumption 4d (Lyapunov condition)] There exist functions $\LY: \Zpos \to \R_+$, $h: \R_+ \to \R_+$ and $f: \Zpos \to \R$ such that $h(k) \downarrow 0$ as $k \to \infty$, $\LY(k) \uparrow \infty$ as $k \to \infty$,
$$
\E f(X^*) = \sum_{k \ge 0} f(k) \pi^*(k) = -\varepsilon < 0,
$$
where $X^* \sim \pi^*$, 
$$
\Exy\left(\LY(Y_1) - \LY(y)\right) \le U
$$
for all $x,y \ge 0$ and for some $U < \infty$, and
$$
\Exy\left(\LY(Y_1) - \LY(y)\right) < f(x) + h(\LY(y))
$$
for all $x \ge 0$ and $y \ge y^*$.

\end{description}

Assumption 4d above is essentially a requirement that the drift of the $Y$ component, when averaged with respect to the stationary distribution of the $X$ component, is negative. As the reader most likely expects, the most common use of this assumption is with $L^Y(y) = y$ for all $y$ and $h(z)=0$ for all $z$, and these are exactly the choices we use in our applications (and in fact even our general result in continuous time assumes that $h(z)=0$ for all $z$). It is however known that linear Lyapunov functions are not always applicable and we thus formulated here a more general result to provide an off-the-shelf technique which may be applied in other situations.

The main result in discrete time is presented in the theorem below.

\begin{theorem} \label{thm:main_discrete}
Under Assumptions 1d-4d above,  $\{(X_n,Y_n)\}_{n \ge 0}$ is positive recurrent.
\end{theorem}

Before proving Theorem \ref{thm:main_discrete}, note that if $y^*=0$, then it is equivalent to \cite[Theorem 1]{foss2013stability} for the process on $\Zpos^2$. Our proof also uses some of the arguments from \cite{foss2013stability}. A generalisation to the case of an arbitrary $y^*$ however is crucial to handle the applications that motivate this paper.

\begin{proof}
We use \cite[Theorem 1]{foss2004overview}. In order to establish the positive recurrence of $\{(X_n,Y_n)\}_{n \ge 0}$, it is sufficient to prove that there exist $0 < \MX, \MY, c< \infty$ and measurable functions $L: \Zpos^2 \to \R_+$ and $T: \Zpos^2 \to \Zpos$ such that
\be \label{eq:fk1}
\sup_{x,y} \frac{T(x,y)}{L(x,y)} < \infty,
\ee
\be \label{eq:fk2}
\sup_{x \le \MX, y \le \MY}\Exy\left(L(X_{T(x,y)},Y_{T(x,y)}) - L(x,y)\right) < \infty
\ee
and
\be \label{eq:fk3}
\Exy\left(L(X_{T(x,y)},Y_{T(x,y)}) - L(x,y)\right) \le -c T(x,y),
\ee
if $x > \MX$ or $y > \MY$. In what follows, we shall choose the appropriate constants and functions so that the above conditions are verified. Note that the function $f$ in Assumption 4d may be taken to be bounded, without loss of generality (see a comment after a similar assumption in \cite{foss2013stability}):
\be \label{eq:aux4}
\sup_{x} |f(x)| \le F < \infty.
\ee

As (due to Assumption 2d) $\{X_n^*\}_{n\ge 0}$ is a positive recurrent chain on $\Zpos$, then the set $\{x \le \MX\}$ is positive recurrent for any $0 < \MX < \infty$. Moreover, there exists a Lyapunov function $\LX: \Zpos \to \R_+$ such that
\be \label{eq:aux8}
\sup_{x \le \MX} \E(\LX(X^*_1) - \LX(x)|X^*_0 = x) < \infty,
\ee
$$
\E(\LX(X^*_1) - \LX(x)|X^*_0 = x) =-1
$$
for all $x > \MX$ and
$$
\lim_{n \to \infty} \frac{1}{n}\sup_{x \le \MX} \E (\LX(X^*_n)|X^*_0 = x) = 0
$$
(see, e.g. \cite{foss2013stability} where this is discussed in detail, and the references therein). The above expressions, along with Assumptions 1d and 3d, imply that
\be \label{eq:aux5}
\sup_{x \le \MX, y \in \Z} \Exy(\LX(X_1) - \LX(x)) < \infty,
\ee
(the finiteness of the supremum over $y > y^*$ follows directly from Assumption 1d and \eqref{eq:aux8} above, and the finiteness of the supremum over $y \le y^*$ follows from Assumption 3d),
\be \label{eq:aux3}
\Exy(\LX(X_1) - \LX(x)) =-1
\ee
for any $x > \MX$ and $y > y^*$, and
\be \label{eq:aux1}
\lim_{n \to \infty} \frac{1}{n}\sup_{x \le \MX, y > y^*} \Exy (\LX(X_n)| A_n) = 0,
\ee
where $A_n = \{Y_k > 0 \quad \text{for all} \quad 0 \le k \le n\}$ is the event that the $Y$ component stays away from $0$ for the first $n$ time steps.

The proof of \cite[Lemma 1]{foss2013stability} may be repeated nearly verbatim to show that Assumption 3d implies that for any $\MX$ there exists $n_0$ such that for any $n \ge n_0$ there exists $\MY_0$ and $\Delta>0$ such that
\be \label{eq:aux2}
\Exy (\LY(Y_n) - \LY(y)) \le - n\Delta
\ee
for all $x \le \MX$ and for all $y$ such that $\LY(y) \ge \MY_0$. As $\LY$ is increasing, the last inequality is equivalent to $y \ge \MY$ for some $\MY$.

We now let $L(x,y) = B\LX(x) + \LY(y)$ with $B > U$ from Assumption 4d. We also choose $0 < \MX < \infty$ arbitrarily and let 
$$T(x,y) = 
\begin{cases}
1, \quad \text{if} \quad (x,y) \in M \quad \text{or} \quad x > \MX, \\
N, \quad \text{if} \quad y \ge \MY, x \le \MX,
\end{cases}
$$
where $N$ and $\MY$ are chosen so that
\be \label{eq:aux6}
\Exy(\LX(X_N)|A_N) < \frac{N \Delta}{2B}
\ee
for all $x \le \MX$ and $y > y^*$ (which is possible thanks to \eqref{eq:aux1}) and
\be \label{eq:aux7}
\Exy (\LY(Y_N) - \LY(y)) \le - N\Delta
\ee
for all $x \le \MX$ and $y > \MY$ (which is possible in light of \eqref{eq:aux2}). We can clearly assume also that $\MY > N H^Y + y^*$ with $H^Y$ from Assumption 3d (this, in particular, guarantees that $\MY > y^*$, so \eqref{eq:aux6} above holds for $y > \MY$).

Condition \eqref{eq:fk1} is clearly satisfied as $N$ is a constant. Let us show that conditions \eqref{eq:fk2} and \eqref{eq:fk3} are satisfied. If $(x,y) \in M$, then
$$
\Exy(L(X_1,Y_1) - L(x,y)) = B\sup_{x \le \MX, y \in \Z} \{\Exy(\LX(X_1) - \LX(x))\} + F + \sup_{y} h(\LY(y)) < \infty,
$$
thanks to \eqref{eq:aux4}, \eqref{eq:aux5} and Assumption 4d. Hence \eqref{eq:fk2} holds.

Consider now $x > \MX$. Then
$$
\Exy(L(X_1,Y_1) - L(x,y)) \le -B + U < 0,
$$
thanks to \eqref{eq:aux3} and our choice of $B$. 

It remains to consider $x \le \MX$ and $y > \MY$. As we have chosen $\MY > N H^Y + y^*$, event $A_N$ happens with probability $1$ and hence, thanks to \eqref{eq:aux6} and \eqref{eq:aux7},
$$
\Exy(L(X_N,Y_N) - L(x,y)) \le N \frac{\Delta}{2} - N \Delta = -N \frac{\Delta}{2}.
$$
The two inequalities above imply that \eqref{eq:fk3} holds with $c = \min\{B-U, \Delta/2\}$.
\end{proof}

\subsection{Continuous time} \label{subsec:continuous}

In this section, we present the assumptions and the main result for continuous time  Markov processes which are, not surprisingly, very similar to the ones obtained for discrete time. We also comment on how the discrete-time results can be easily applied to prove stability in continuous time. 

Assume that $\{(X_t,Y_t)\}_{t \ge 0}$ is a time-homogeneous and irreducible Markov process on $\Zpos^2$. We denote by
$q_{(x,y)}(x',y')$ the rate of transition from $(x,y)$ to $(x',y')$ and denote by
$$
q_{(x,y)}(x') = \sum_{y'} q_{(x,y)}(x',y') = \lim_{h \downarrow 0} \frac{\Prob(X(h)=x'|X(0)=x,Y(0)=y)}{h}.
$$

We make the following additional assumptions.
\begin{description}
\item[Assumption 1c (eventual homogeneity in continuous time)] There exists $0 \le y^* < \infty$ such that
$$
\qxy(x') = q^*(x,x')
$$
and does not depend $y$ as long as $y \ge y^*$.

\item[Assumption 2c   (positive recurrence of the $X$ component in continuous time)] Introduce a Markov process $\{X_t^*\}_{t \ge 0}$ such that its transition rates from $x$ to $x'$ are given by $q^*(x,x')$ for any $x,x' \in \Zpos$. We assume that $\{X_t^*\}_{t \ge 0}$ is positive recurrent and denote its stationary distribution by $\pi^*$.

\item[Assumption 3c (boundedness of jumps in continuous time)] There exists $0 < H^Y < \infty$ such that $\qxy(x',y') = 0$ if $y' < y - H^Y$.

There also exists $0 \le H^X < \infty$ such that $\qxy(x',y') = 0$ if $x'-x > H^X$ and $y < y^*$.

\item[ Assumption 4c (Lyapunov condition in continuous time)] There exist functions $\LY: \Zpos \to \R_+$ and $f: \Zpos \to \R$ such that $\LY(k) \uparrow \infty$ as $k \to \infty$,
\be \label{eq:ass4_1}
\E f(X^*) = \sum_{x \ge 0} f(x) \pi^*(x) = -\varepsilon < 0,
\ee
where $X^* \sim \pi^*$, 
\be \label{eq:ass4_2}
\sum_{x',y'} \qxy(x',y') \left(\LY(y') - \LY(y)\right) \le U
\ee
for all $x,y \ge 0$ and for some $U < \infty$, and
\be \label{eq:ass4_3}
\sum_{x',y'} \qxy(x',y') \left(\LY(y') - \LY(y)\right) < f(x)
\ee
for all $x\ge 0$ and $y \ge y^*$.

\item[Assumption 5c (non-explosiveness)]
Let
$$
v(x,y)= \sum_{x',y'} \qxy(x',y').
$$
We assume that transition rates are such that
\be \label{eq:ass5_1}
0 < \inf_{(x,y)} v(x,y) \le \sup_{(x,y)} v(x,y) < \infty.
\ee

The RHS of Assumption 5c is a standard one guaranteeing that the process is non-explosive. The LHS is a technical condition required in our proof.
\end{description}

We note also that Assumption 1c implies that for any $y \ge y^*$
\be \label{eq:ass5_2}
v(x,y) =  \sum_{x'} \sum_{y'} \qxy(x',y') = \sum_{x'} q^*(x,x') = v^*(x)
\ee
does not depend on $y$.

The main result in continuous time is presented in the following theorem.

\begin{theorem} \label{thm:main_continuous}
Under Assumptions 1c-5c, $\{(X_t,Y_t)\}_{t \ge 0}$ is positive recurrent.
\end{theorem}

\begin{proof}
We only comment on the proof as it is straightforward and follows a standard approach of showing that the jump chain of the Markov process is positive recurrent (for a recent nice presentation of these ideas see, e.g., \cite[Appendix D]{kelly2014stochastic}). Consider the jump chain $\{(X^J_n,Y^J_n)\}_{n \ge 0}$ of the process $\{(X_t,Y_t)\}_{t \ge 0}$, which consists of the values of the process observed at points when a jump occurs. It is clear that the transition probabilities for the jump chain are given by
$$
\Prob\left((X^J_{n+1},Y^J_{n+1}) = (x',y')|(X^J_n,Y^J_n) = (x,y)\right) = \frac{\qxy(x',y')}{v(x,y)},
$$
and it is easy to see that Assumptions 1c-3c imply that Assumptions 1d-3d for the discrete time model hold. Note that
$$
p^*(x,x') = \frac{q^*(x,x')}{v(x)}
$$
and the stationary distribution of the jump chain for $X^*$, say $\tilde{\pi}^*$, is given by
$$
\tilde{\pi}^*(x) = \frac{\pi^*(x) v^*(x)}{\sum_{z} \pi^*(z) v^*(z)}.
$$
Assume now Assumption 4c holds with functions $L^Y$ and $f$. For the jump chain, 
$$
\Exy(L^Y(Y_1) - L^Y(y)) = \sum_{x',y'} \frac{\qxy(x',y')}{v(x,y)} (L^Y(y') - L^Y(y))  \le \frac{U}{v(x,y)} \le \tilde{U}
$$
for some $\tilde{U} < \infty$ thanks to Assumptions 4c and 5c.
Further, for all $y \ge y^*$,
$$
\Exy(L^Y(Y_1) - L^Y(y)) = \sum_{x',y'} \frac{\qxy(x',y')}{v(x)} (L^Y(y') - L^Y(y)) < \frac{f(x)}{v(x)} = \tilde{f}(x),
$$
say. Then
$$
\sum_x \tilde{f}(x) \tilde{\pi}^*(x) = \frac{\sum_x f(x) \pi^*(x)}{\sum_{z} \pi^*(z) v^*(z)} < 0,
$$
thanks to Assumption 4c and hence Assumption 4d holds for the jump chain with function $L^Y$, $\tilde{f}$ and $h(z)=0$ for all $z$.

Theorem \ref{thm:main_discrete} then implies the positive recurrence of the jump chain which, in turn, implies the positive recurrence of the Markov process.
\end{proof}

\subsection{Application: QBD processes in the quarter-plane} \label{subsec:qbd}

Let us consider a so-called level-dependent QBD process on $\Zpos^2$ (with the first coordinate representing the phase and the second coordinate representing the level) so that
\be \label{eq:def_qbd1}
q_{(i,0)}(i',0) = B_0 (i,i')
\ee
\begin{align} 
q_{(i,k)}(i',k+1) &= A_1(i,i'),\label{eq:def_qbd2}\\
q_{(i,k)}(i',k) &= A_0(i,i'),\label{eq:def_qbd3}
\end{align}
for $k \ge 0$, and
\be \label{eq:def_qbd4}
q_{(i,k)}(i',k-1) = A_{-1}(i,i'),
\ee
for $k \ge 1$, with some (infinite in general) matrices $A_1, A_0, A_{-1}$ and $B_0$. Diagonal entries of $A_0$ and $B_0$ are such that the matrix has row sums equal to $0$ (see, e.g. \cite{latouche2003drift}).

\begin{figure}[h!]
\begin{center}
\begin{tikzpicture}[font = \scriptsize]

\def \mmax {8.5}
\def \nmax {5.5}

\draw[black!40] (0,0) -- (\mmax,0);
\draw[black!40] (0,0) -- (0,\nmax);

\node[right] (m_label) at (\mmax,0) {phase};
\node[right] at (\mmax-0.7,-0.5) {$\cdots$};
\node[right] at (\mmax-3.2,-0.5) {$i+1\ \cdots\ i+H^X$};
\node[right] at (\mmax-3.7,-0.5) {$i$};
\node[right] at (\mmax-4.8,-0.5) {$i-1$};
\node[right] at (\mmax-5.8,-0.5) {$i-2$};
\node[right] at (0,-0.5) {$0$};

\node[above] (n_label) at (0,\nmax) {level};
\node[above] at (-0.5,\nmax-1) {$k+1$};
\node[above] at (-0.5,\nmax-2) {$k$};
\node[above] at (-0.5,\nmax-3) {$k-1$};
\node[above] at (-0.5,0) {$0$};
\node[above] at (-0.5,1) {$1$};

\foreach \mm in {0,3,4,5,6,7.2}{
    \foreach \nn in {0,1,3,4,5}{
    \node[draw, circle, minimum size = 2, fill = black, inner sep = 0] (\mm-\nn) at (\mm,\nn) {};
    }
}

 \foreach \nn in {0,1,3,4,5}{
    \node at (1.5,\nn) {$\cdots$};
    \node at (6.5,\nn) {$\cdots$};
    \node at (\mmax-0.4,\nn) {$\cdots$};
    }
 \foreach \mm in {0,3,4,5,6,7.2}{
    \node at (\mm,2) {$\vdots$};
    }

\draw[-latex] (5,4) to [out=35,in=190]  (7.2,5) ;
\draw[-latex] (5,4) to [out=20,in=160]  (7.2,4) ;
\draw[-latex] (5,4) to   (6,5) ;
\draw[-latex] (5,4) to  (6,4) ;
\draw[-latex] (5,4) -- (5,5) node[pos = 1,above] { };
\draw[-latex] (5,4) -- (4,5) node[pos = 0.7,left] { };
\draw[-latex] (5,4) -- (4,4) node[pos = 0.7,above] { };
\draw[-latex] (5,4) to [out=145,in=0]  (3,5) ;
\draw[-latex] (5,4) to [out=165,in=20]   (3,4) ;
\draw[-latex] (5,4)  to [out=160,in=20] (0,4);
\draw[-latex] (5,4)  to [out=160,in=0] (0,5);

\draw[-latex] (5,4) to [out=-35,in=190]  (7.2,3) ;
\draw[-latex] (5,4) to   (6,3) ;
\draw[-latex] (5,4) -- (5,3) node[pos = 1,above] { };
\draw[-latex] (5,4) -- (4,3) node[pos = 0.7,left] { };
\draw[-latex] (5,4) to [out=-145,in=0]  (3,3) ;
\draw[-latex] (5,4)  to [out=-160,in=0] (0,3);

\draw[-latex] (5,0) to [out=35,in=190]  (7.2,1) ;
\draw[-latex] (5,0) to [out=20,in=160]  (7.2,0) ;
\draw[-latex] (5,0) to   (6,1) ;
\draw[-latex] (5,0) to  (6,0) ;
\draw[-latex] (5,0) -- (5,1) node[pos = 1,above] { };
\draw[-latex] (5,0) -- (4,1) node[pos = 0.7,left] { };
\draw[-latex] (5,0) -- (4,0) node[pos = 0.7,above] { };
\draw[-latex] (5,0) to [out=145,in=0]  (3,1) ;
\draw[-latex] (5,0) to [out=165,in=20]   (3,0) ;
\draw[-latex] (5,0)  to [out=160,in=20] (0,0);
\draw[-latex] (5,0)  to [out=160,in=0] (0,1);
\end{tikzpicture} 
\end{center}
\caption{Schematic overview of the transition rate diagram of the level-dependent QBD process on $\Zpos^2$ \label{fig:QBDfigure}}
\end{figure}
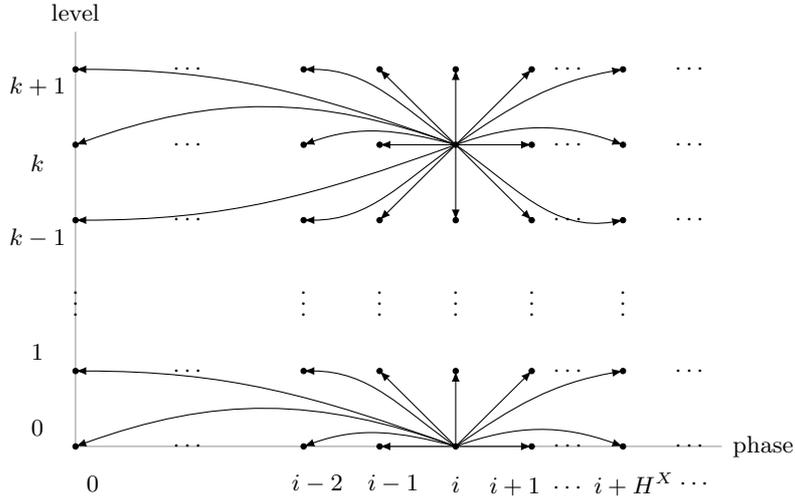
\begin{remark}
One can readily see that, for $H^X=1$, the stochastic process depicted in  Figure \ref{fig:QBDfigure} can be interpreted as a (level-independent) GI/M/1-type Markov chain by considering the first coordinate to be the level and the second coordinate to be the phase. This will lead to a structure similar to the one observed in \cite{ramaswami1996some, latouche2013level}. As both the level and the phase are infinite, we have opted for the simpler notation  inherited to the QBD structure rather than the more intricate notation inherited to the GI/G/1-type structure. As such, we opted to refer to the first coordinate as the phase and refer to the second coordinate as the level.
\end{remark}

Note again that the standard assumption is that the number of phases is bounded, i.e. $i,i' \le m$ for some $m<\infty$. {\it We do not make such assumption.} We only need to assume that the transitions to the right in the phase are bounded by $H^X$ say, i.e.
$$
A_{-1}(i,i') = A_0(i,i') = A_1(i,i') =  B_0(i,i') = 0 \quad \text{if} \quad i-i'>H^X
$$
(note we allow unbounded jumps to the left).

One can then see that if we take $X_t$ to represent the phase and $Y_t$ to represent the level, then Assumption 1c is automatically satisfied with $y^*=1$. Assumption 3c is satisfied with $H^Y=1$ and $H^X$ defined above.

In the QBD notation above,
$$
q^*(i,i') = A_{-1}(i,i') + A_0(i,i') + A_1(i,i') 
$$
and hence the stationary distribution $\pi^*$ from Assumption 2c, if it exists, satisfies the equation
$$
\pi^*(A_{-1} + A_0 + A_1 ) = 0
$$
(note this distribution always exists if the number of phases is bounded, but in general one needs to check this, hence Assumption 2c still needs to be accounted for). Consider now Assumption 4c. Take $\LY(y)=y$, then \eqref{eq:ass4_2} is satisfied automatically. With
$$
f(i) = \sum_{i'} (A_1(i,i') - A_{-1}(i,i'))
$$
relation \eqref{eq:ass4_3} is also satisfied automatically, and one only needs to check whether or not \eqref{eq:ass4_1} holds. This is equivalent to
$$
\pi^*(A_1-A_{-1})\mathbf{1} < 0,
$$
where $\mathbf{1}$ is an identity vector. 

We now need to make sure that Assumption 5c holds.
In the QBD notation,
$$
v(i,k) = \begin{cases}
\sum_{i'} (A_{-1}(i,i') + A_0(i,i') + A_1(i,i')), \quad k \ge 1,\\
\sum_{i'} A_1(i,i') + B_0(i,i'), \quad k=0,
\end{cases}
$$
and we need to assume that the infimum of this function is positive and its supremum is finite (this of course holds automatically if the number of phases is finite).

To summarise, our results can be directly applied to the QBD processes defined in 
Equations \eqref{eq:def_qbd1}--\eqref{eq:def_qbd4}
and they are presented in the following theorem.

\begin{theorem} \label{thm:qbd}
Consider a QBD process with transitions as defined in Equations \eqref{eq:def_qbd1} -- \eqref{eq:def_qbd4} and assume that there exists $0 \le H^X < \infty$ such that
$$
A_{-1}(i,i') = A_0(i,i') = A_1(i,i') =B_0(i,i') = 0 \quad \text{if} \quad  i'-i > H^X.
$$
Assume also that
$$
0< \inf_{i,k} v(i,k) \le \sup_{i,k} v(i,k) < \infty
$$
for the function $v$ defined above. Assume further that there exists $\pi^*$ that satisfies
$$
\pi^*(A_{-1} + A_0 + A_1) = 0.
$$
If
$$
\pi^*(A_1-A_{-1})\mathbf{1} < 0,
$$
then the QBD process is positive recurrent.
\end{theorem}

\begin{remark}
One can readily see that our results may be applied to a more general process where jumps of a size larger than $1$ (but bounded) are allowed for levels (i.e. where along with matrices $A_{-1}$, $A_0$ and $A_1$ for transitions to one level below, same level and one level above, respectively, one can introduce matrices $A_m$, $-H^Y \le m \le H^Y$ for transitions that include a jump of size $m$ in level).
\end{remark}

\section{Generalisations} \label{sec: generalisations}

In this section, we discuss generalisations of our main results which do not require new ideas in their proofs. We therefore leave them without proofs. We start with a generalisation that immediately leads to a stability result for QBD-type processes on a half-line. We present these results for a particular case of a new model, and we discuss this in detail. In the following subsection, we discuss further generalisations.

\subsection{Processes on the half-plane} \label{subsec:binomial}

We start by presenting a simple generalisation of our results to a time-homogeneous and irreducible Markov process on $\Zpos \times Z$ whose transitions in the first component do not depend on the position of the second component as long as the latter is either positive or negative (but may be different in the two cases). More precisely, we make the following assumptions which are essentially our assumptions in continuous time for both the positive and negative values of the $Y$ component.
\begin{description}
\item[Assumption 1g]
$$
\sum_{y'} \qxy(x',y') = q_+^*(x,x'),
$$
which does not depend $y$ as long as $y > 0$ and
$$
\sum_{y'} \qxy(x',y') = q_-^*(x,x'),
$$
which does not depend $y$ as long as $y < 0$.

\item[Assumption 2g] Introduce Markov processes $\{X_{+,t}^*\}_{t \ge 0}$ $\{X_{-,t}^*\}_{t \ge 0}$ such that their transition rates from $x$ to $x'$ are given by $q_+^*(x,x')$ and $q_-^*(x,x')$. We assume that both these processes are positive recurrent and denote their stationary distributions by $\pi_+^*$ and $\pi_-^*$, respectively.

\item[Assumption 3g] There exists $0 < H^Y < \infty$ such that $\qxy(x',y') = 0$ if $y > 0$ and $y' < y - H^Y$ or $y < 0$ and $y' > y + H^Y$.

There also exists $0 \le H^X < \infty$ such that $\qxy(x',y') = 0$ if $x'-x > H^X$ and $y =0$.

\item[Assumption 4g] There exist functions $\LY_-, L^Y_+: \Zpos \to \R_+$ and $f_+,f_-: \Zpos \to \R$ such that $\LY_{\pm}(k) \uparrow \infty$ as $k \to \infty$,
\be \label{eq:ass4_1e}
\E f_{\pm}(X_{\pm}^*) = \sum_{x \ge 0} f_{\pm}(x) \pi_{\pm}^*(x) < 0,
\ee
where $X_{\pm}^* \sim \pi_{\pm}^*$, 
\be \label{eq:ass4_2e}
\sum_{x',y'} \qxy(x',y') \left(\LY_{\pm}(y') - \LY_{\pm}(y)\right) \le U
\ee
for all $x,y \ge 0$ and for some $U < \infty$, and
\be \label{eq:ass4_3e}
\sum_{x',y'} \qxy(x',y') \left(\LY_{\pm}(y') - \LY_{\pm}(y)\right) < 
f_{\pm}(x),
\ee
where index $+$ is chosen if $y > 0$ and index $-$ is chosen if $y < 0$.

\item[Assumption 5g] Transition rates are such that Assumption 5c holds for both positive and negative $y$.
\end{description}
It is easy to see that if Assumptions 1g-5g hold, the proof of Theorem \ref{thm:main_continuous} may be applied in a straightforward manner (with an appropriate analogue of Theorem \ref{thm:main_discrete} for Markov chains on $\Zpos \times \Z$) to show positive recurrence of the process described above.

The result above presents a method for demonstrating the stability of QBD-type processes on a half-plane. We proceed  by  applying it to a particular example, namely a two-dimensional variation of the queueing system with binomial catastrophes studied in \cite{kapodistria2011m}. More concretely, assume there are two queues, say $Q_i$,  $i=1,2$, with arrival rates $\lambda_i$ and service rates $\mu_i$, $i=1,2$, respectively. There are also coupled abandonments (or binomial catastrophes) that happen at rate $\gamma$. That is, if at the time of a catastrophe there are $q_1$ customers in queue $1$ and $q_2$ customers in queue $2$, then a random number $N \sim Bin(q,p)$ leaves from both queues, where $q=\min\{q_1,q_2\}$ and $0 < p \le 1$ is a parameter of the system.

The stability question is trivial for the corresponding one-dimensional system as any of the two queues would be stable if it were in isolation. Denote by $\pi^{(1)}$ and $\pi^{(2)}$ stationary distributions of queues $1$ and $2$, respectively, in isolation.

With notation $Q_1(t)$ and $Q_2(t)$ for the states of the two queues at time $t$, $t\ge 0$, consider 
$$\left(X(t), Y(t)\right) = \left(\min\{Q_1(t),Q_2(t)\}, Q_1(t) - Q_2(t)\right),$$ which is a Markov process, and it is positive recurrent if and only if the original process is positive recurrent. One can also see that, as long as $Y(t) > 0$, the transitions of $X(t)$ are exactly equal to the transitions of queue $2$ in isolation (and have $\pi^{(2)}$ as stationary distribution). Similarly, as long as $Y(t) < 0$, the transitions of $X(t)$ are exactly equal to the transitions of queue $1$ in isolation (and have $\pi^{(1)}$ as stationary distribution).

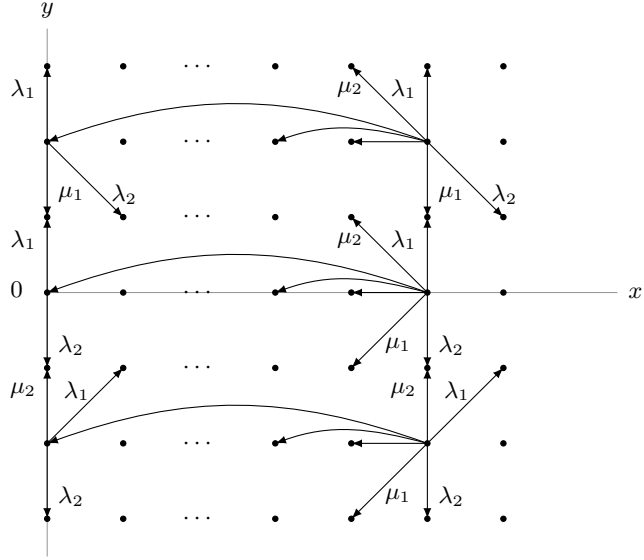
\begin{figure}[h!]
\begin{center}
\begin{tikzpicture}[font = \scriptsize]

\def \mmax {7.5}
\def \nmax {6.5}

\draw[black!40] (0,3) -- (\mmax,3);
\draw[black!40] (0,-0.5) -- (0,\nmax);

\node[right] (m_label) at (\mmax,3) {$x$};

\node[above] (n_label) at (0,\nmax) {$y$};
\node[above] at (-0.4,2.8) {$0$};

\foreach \mm in {0,,3,4,5,6}{
    \foreach \nn in {0,1,2,3,4,5,6}{
    \node[draw, circle, minimum size = 2, fill = black, inner sep = 0] (\mm-\nn) at (\mm,\nn) {};
    }
}

 \foreach \nn in {0,1,2,3,4,5,6}{
    \node at (2,\nn) {$\cdots$};
    }

\draw[-latex] (5,5) -- (5,6) node[pos = 0.7,left] { $\lambda_1$};
\draw[-latex] (5,5) -- (4,6) node[pos = 0.7,left] {$\mu_2$};
\draw[-latex] (5,5) -- (4,5) node[pos = 0.7,below] {};
\draw[-latex] (5,5) to [out=165,in=20]   (3,5) ;
\draw[-latex] (5,5)  to [out=160,in=20] (0,5);
\draw[-latex] (5,5) -- (6,4) node[pos = 0.7,right] {$\lambda_2$};
\draw[-latex] (5,5) -- (5,4) node[pos = 0.7,right] {$\mu_1$};
\draw[-latex] (0,5) -- (0,6) node[pos = 0.7,left] { $\lambda_1$};
\draw[-latex] (0,5) -- (1,4) node[pos = 0.7,right] {$\lambda_2$};
\draw[-latex] (0,5) -- (0,4) node[pos = 0.7,right] {$\mu_1$};

\draw[-latex] (5,3) -- (5,4) node[pos = 0.7,left] {$\lambda_1$};
\draw[-latex] (5,3) -- (4,4) node[pos = 0.7,left] {$\mu_2$};
\draw[-latex] (5,3) -- (4,3) node[pos = 0.7,above] { };
\draw[-latex] (5,3) to [out=165,in=20]   (3,3);
\draw[-latex] (5,3)  to [out=160,in=20] (0,3);
\draw[-latex] (5,3) -- (4,2) node[pos = 0.7,right] {$\mu_1$};
\draw[-latex] (5,3) -- (5,2) node[pos = 0.7,right] {$\lambda_2$};
\draw[-latex] (0,3) -- (0,4) node[pos = 0.7,left] {$\lambda_1$};
\draw[-latex] (0,3) -- (0,2) node[pos = 0.7,right] {$\lambda_2$};

\draw[-latex] (5,1) -- (5,2) node[pos = 0.7,left] {$\mu_2$};
\draw[-latex] (5,1) -- (4,0) node[pos = 0.7,right] {$\mu_1$};
\draw[-latex] (5,1) -- (4,1) node[pos = 0.7,above] {};
\draw[-latex] (5,1) to [out=165,in=20]   (3,1) ;
\draw[-latex] (5,1)  to [out=160,in=20] (0,1);
\draw[-latex] (5,1) -- (6,2) node[pos = 0.7,left] {$\lambda_1$};
\draw[-latex] (5,1) -- (5,0) node[pos = 0.7,right] {$\lambda_2$};
\draw[-latex] (0,1) -- (0,2) node[pos = 0.7,left] {$\mu_2$};
\draw[-latex] (0,1) -- (1,2) node[pos = 0.7,left] {$\lambda_1$};
\draw[-latex] (0,1) -- (0,0) node[pos = 0.7,right] {$\lambda_2$};

\end{tikzpicture} 
\end{center}
\caption{Schematic overview of the transition rate diagram of the two-dimensional binomial catastrophes model  on the half-plane}
\end{figure}

It is easy to see that Assumptions 1g-3g and 5g above hold with obvious constants and transitions, and the stability condition will be driven by Assumption 4g. With $L_{\pm}^Y(y) = \pm y$
$$
\sum_{x',y'} \qxy(x',y') \left(\LY(y') - \LY(y)\right) = 
\begin{cases}
(\lambda_1 - \mu_1) - (\lambda_2 - \mu_2), \quad x>0, y>0\\
(\lambda_2 - \mu_2) - (\lambda_1 - \mu_1), \quad x>0, y<0,\\
(\lambda_1 - \mu_1) - \lambda_2, \quad y > 0, x=0,\\
(\lambda_2 - \mu_2) - \lambda_1, \quad y < 0, x=0,\\
\end{cases}
$$
and hence Equation \eqref{eq:ass4_3e} is satisfied with functions
$$
f_+(x) = 
\begin{cases}
(\lambda_1 - \mu_1) - (\lambda_2 - \mu_2), \quad x>0,\\
(\lambda_1 - \mu_1) - \lambda_2, \quad x=0,
\end{cases}
$$
$$
f_-(x) = 
\begin{cases}
(\lambda_1 - \mu_1) - (\lambda_2 - \mu_2), \quad x>0,\\
\lambda_1 - (\lambda_2 - \mu_2), \quad x=0.
\end{cases}
$$
Stability is then ensured, due to \eqref{eq:ass4_1e} as long as
$$
((\lambda_1 - \mu_1) - (\lambda_2 - \mu_2))(1-\pi^{(2)}(0)) + ((\lambda_1 - \mu_1) - \lambda_2) \pi^{(2)}(0) < 0
$$
and
$$
((\lambda_2 - \mu_2) - (\lambda_1 - \mu_1))(1-\pi^{(1)}(0)) + ((\lambda_2 - \mu_2) - \lambda_1)\pi^{(1)}(0) < 0
$$
or, equivalently,
$$
\lambda_1 - \mu_1 < \lambda_2 - \mu_2(1-\pi^{(2)}(0))
$$
and
$$
\lambda_2 - \mu_2 < \lambda_1 - \mu_1(1-\pi^{(1)}(0)).
$$
Note that, interestingly, if $\lambda_1 - \mu_1 = \lambda_2 - \mu_2$, the system is always stable.
\subsection{Further generalisations} \label{subsec:further_gen}

The main goal of this paper has been to provide a method to determine stability for QBD-type processes on the quarter-plane and on the half-plane. To this end, we have chosen to present our general results for exactly such state spaces. One can see however that, for instance, Theorem \ref{thm:main_discrete} may be proven for a much wider class of processes on general normed vector spaces in a way similar to conditions (A) and (B) in \cite{foss2013stability}. As mentioned above, these results do not present any technical difficulties but require a rather complex notation and further cumbersome derivations. We leave these generalisations to the reader. 

\section*{Acknowledgements}
The authors thank  Zbigniew Palmowski (Department of Applied Mathematics, Faculty of Pure and Applied Mathematics, Wrocław University of Science and Technology) and Peter G. Taylor (School of Mathematics and Statistics,  University of Melbourne) for discussions on queueing theory applications and further generalisations of this work.

The work of Stella Kapodistria is supported by the Netherlands Organisation for Scientific Research (NWO) through the Gravitation-grant NETWORKS-024.002.003.

Seva Shneer thanks Eindhoven University of Technology for hospitality during a number of his recent visits when a part of this work was conducted.

\bibliographystyle{plain}
\bibliography{stab_bib}

\begin{thebibliography}{10}

\bibitem{adan2020local}
Ivo Adan, Sergey Foss, Seva Shneer, and Gideon Weiss.
\newblock Local stability in a transient {M}arkov chain.
\newblock {\em Statistics \& Probability Letters}, page 108855, 2020.

\bibitem{fayolle1999random}
Guy Fayolle, Roudolf Iasnogorodski, and Vadim Malyshev.
\newblock {\em Random {W}alks in the {Q}uarter-plane}, volume~40.
\newblock Springer, 1999.

\bibitem{foss2013stability}
Sergey Foss, Seva Shneer, and Andrey Tyurlikov.
\newblock Stability of a {M}arkov-modulated {M}arkov chain, with application to
  a wireless network governed by two protocols.
\newblock {\em Stochastic Systems}, 2(1):208--231, 2013.

\bibitem{foss2004overview}
Serguei Foss and Takis Konstantopoulos.
\newblock An overview of some stochastic stability methods.
\newblock {\em Journal of the Operations Research Society of Japan},
  47(4):275--303, 2004.

\bibitem{haque2005sufficient}
Lani Haque, Yiqiang~Q Zhao, and Liming Liu.
\newblock Sufficient conditions for a geometric tail in a {QBD} process with
  many countable levels and phases.
\newblock {\em Stochastic Models}, 21(1):77--99, 2005.

\bibitem{kapodistria2011m}
Stella Kapodistria.
\newblock The {M}/{M}/1 queue with synchronized abandonments.
\newblock {\em Queueing Systems}, 68(1):79--109, 2011.

\bibitem{kapodistria2017matrix}
Stella Kapodistria and Zbigniew Palmowski.
\newblock Matrix geometric approach for random walks: Stability condition and
  equilibrium distribution.
\newblock {\em Stochastic Models}, 33(4):572--597, 2017.

\bibitem{kelly2014stochastic}
Frank Kelly and Elena Yudovina.
\newblock {\em Stochastic Networks}, volume~2.
\newblock Cambridge University Press, 2014.

\bibitem{latouche2013level}
Guy Latouche, Safieh Mahmoodi, and Peter~G Taylor.
\newblock Level-phase independent stationary distributions for {GI}/{M}/1-type
  {M}arkov chains with infinitely-many phases.
\newblock {\em Performance Evaluation}, 70(9):551--563, 2013.

\bibitem{latouche2003drift}
Guy Latouche and Peter~G Taylor.
\newblock Drift conditions for matrix-analytic models.
\newblock {\em Mathematics of Operations Research}, 28(2):346--360, 2003.

\bibitem{motyer2006decay}
Allan~J Motyer and Peter~G Taylor.
\newblock Decay rates for quasi-birth-and-death processes with countably many
  phases and tridiagonal block generators.
\newblock {\em Advances in Applied Probability}, 38(2):522--544, 2006.

\bibitem{neuts1981matrix}
Marcel~F Neuts.
\newblock {\em Matrix-geometric {S}olutions in {S}tochastic {M}odels -- {A}n
  {A}lgorithmic {A}pproach}.
\newblock Johns Hopkins University Press, 1981.

\bibitem{ramaswami1996some}
V~Ramaswami and Peter~G Taylor.
\newblock Some properties of the rate operators in level dependent
  quasi-birth-and-death processes with countable number of phases.
\newblock {\em Stochastic {M}odels}, 12(1):143--164, 1996.

\bibitem{takahashi2001geometric}
Yukio Takahashi, Kou Fujimoto, and Naoki Makimoto.
\newblock Geometric decay of the steady-state probabilities in a
  quasi-birth-and-death process with a countable number of phases.
\newblock {\em Stochastic Models}, 17(1):1--24, 2001.

\bibitem{tweedie1982operator}
Richard~L Tweedie.
\newblock Operator-geometric stationary distributions for {M}arkov chains, with
  application to queueing models.
\newblock {\em Advances in Applied Probability}, 14(2):368--391, 1982.

\end{thebibliography}

\end{document}